\documentclass[a4paper,11pt]{article}

\usepackage[left=2.5cm,right=2.5cm,top=3cm,bottom=3cm,pdftex]{geometry}

\usepackage{amssymb}
\usepackage{amsmath}
\usepackage{amsthm}

\usepackage{dsfont}
\usepackage{url}
\usepackage{natbib}

\usepackage[utf8]{inputenc}
\usepackage[T1]{fontenc}

\usepackage[pdftex]{graphicx}
\DeclareGraphicsExtensions{.png, .pdf, .jpg} 

\usepackage[pdftex, colorlinks, linkcolor=blue, urlcolor=blue, citecolor=blue, breaklinks=true]{hyperref}

\newtheorem{thm}{Theorem}
\newtheorem{lem}{Lemma}
\newtheorem{fol}{Corollary}
\newtheorem{bem}{Remark}

\begin{document}

\begin{center}
\Large
\textbf{Equivalence theorems for compound design problems with application in mixed models}
\end{center}

\begin{center}
\textbf{Maryna Prus}
\end{center}

\begin{abstract}
In the present paper we consider design criteria which depend on several designs simultaneously. We formulate equivalence theorems based on moment matrices (if criteria depend on designs via moment matrices) or with respect to the designs themselves (for finite design regions). We apply the obtained optimality conditions to the multiple-group random coefficient regression models and illustrate the results by simple examples.
\end{abstract}


\section{Introduction}

The subject of this work is compound design problems - optimization problems with optimality criteria depending on several designs simultaneously. Such optimality criteria can be, for example, commonly used design criteria for estimation of unknown model parameters in case when the covariance matrix of the estimation depends on several designs (see e.\,g. \cite{fedo}, \cite{schm1}). For such criteria general equivalence theorem proposed in \cite{kie} cannot be used directly. In \cite{fedo} optimal designs were obtained for specific regression functions. In \cite{schm1} particular group-wise identical designs have been discussed.

In this paper we formulate equivalence theorems for two kinds of compound design problems: 1) problems on finite experimental regions and 2) problems with optimality criteria depending on designs via moment (or information) matrices. For both cases we assume the optimality criteria to be convex and differentiable in the designs themselves or the moment matrices, respectively. In case 1) we formulate optimality conditions with respect to the designs directly (as proposed in \cite{whi} for one-design problems). These results can be useful in situations when design criteria cannot be presented as functions of moment matrices (see e.\,g. \cite{bos}). In case 2) optimality conditions are formulated with respect to the moment matrices. Therefore, no additional restrictions of the experimental regions are needed. 

We apply the equivalence theorems to multiple-group random coefficient regression (RCR) models. In these models observational units (individuals) are assigned to several groups. Within one group same designs (group-designs) for all individuals have been assumed. Group-designs for individuals from different groups are in general not the same.
Most of commonly used design criteria in multiple-group RCR models are functions of several group-designs. 
The particular case of these models with one observation per individual has been considered in \cite{gra}. In \cite{pru3}, ch.~6, models with group-specific mean parameters were briefly discussed. \cite{blu}, \cite{kun}, \cite{lem} and \cite{pru7} considered models with particular regression functions and specific covariance structure of random effects. In \cite{ent} and \cite{pru1} same design for all observational units have been assumed.

The paper has the following structure: Section~2 provides equivalence theorems for the compound design problems. In Section~3 we apply the obtained optimality conditions to the multiple-group RCR models. In Section~4 we illustrate the results by a simple example. The paper is concluded by a short discussion in Section~5.

\section{Optimality Conditions for Compound Design Problems}

We consider a compound design problem in which $\xi_1, \dots, \xi_s$ are probability measures (designs)  on experimental regions $\mathcal{X}_1, \dots, \mathcal{X}_s$, respectively, and
$\phi$ is a design criterion which depends on $\xi_1, \dots, \xi_s$ simultaneously and has to be minimized. $\Xi_i$ denotes the set of all designs on $\mathcal{X}_i$, $i=1, \dots, s$. For any $x_i \in \mathcal{X}_i$, $\delta_{x_i}$ denotes the particular design $\xi_i$ with all observations at point $x_i$.
For convenience we use the notation $\mbox{\boldmath{$\xi$}}=(\xi_1, \dots \xi_s)$ for a vector of designs $\xi_i \in \Xi_i$, $i=1, \dots, s$. Then $\mbox{\boldmath{$\xi$}}\in \Xi$ for $\Xi=\times_{i=1}^s\Xi_i$, where ''$\times$'' denotes the Cartesian product.


In Section~\ref{21} we consider compound design problems, where all design regions are assumed to be finite. We formulate an equivalence theorem (Theorem~\ref{oc2}) with respect to the designs directly.

In Section~\ref{22} we consider design criteria depending on designs via moment matrices and we propose an equivalence theorem (Theorem~\ref{oc1}) based on this structure. In this case no additional restrictions of the experimental regions are needed.

\subsection{Optimality conditions in case of finite design regions}\label{21}

In this section we restrict ourselves on optimization problems on finite design regions: $|\mathcal{X}_i|=k_i<\infty$ for all $i=1, \dots, s$. 
$\phi: \Xi \rightarrow (-\infty;\infty]$ denotes a design criterion.
We use the notation $\Phi(\mbox{\boldmath{$\xi$}},\tilde{\mbox{\boldmath{$\xi$}}})$ for the directional derivative of $\phi$ at $\mbox{\boldmath{$\xi$}}$ in direction of $\tilde{\mbox{\boldmath{$\xi$}}}$:
\begin{equation}\label{dd1}
\Phi(\mbox{\boldmath{$\xi$}},\tilde{\mbox{\boldmath{$\xi$}}})=\lim_{\alpha\, \searrow\, 0}\frac{1}{\alpha}\left(\phi((1-\alpha)\mbox{\boldmath{$\xi$}}+\alpha\tilde{\mbox{\boldmath{$\xi$}}})-\phi(\mbox{\boldmath{$\xi$}})\right).
\end{equation}
Further we define the partial directional derivative of $\phi$ at $\xi_i$ in direction of $\tilde{\xi}_i$ 
as follows:
\begin{equation}\label{pdd1}
\Phi_{\xi_{i'}, i'\neq i}(\xi_i,\tilde{\xi}_i)=\Phi(\mbox{\boldmath{$\xi$}},\breve{\mbox{\boldmath{$\xi$}}}),
\end{equation}
where $\breve{\mbox{\boldmath{$\xi$}}}=(\breve{\xi_1}, \dots \breve{\xi_s})$ with
 $\breve{\xi}_{i'}=\xi_{i'}$, $i'=1, \dots, s$, $i'\neq i$, and $\breve{\xi}_{i}=\tilde{\xi}_{i}$.

\begin{thm}\label{oc2}
Let $\phi: \Xi \rightarrow (-\infty;\infty]$ be convex and differentiable. 
\begin{enumerate}
	\item[a)] The following statements are equivalent:
	\begin{enumerate}
	\item[(i)] $\mbox{\boldmath{$\xi$}}^*=(\xi_1^*, \dots, \xi_s^*)$ minimizes $\phi(\mbox{\boldmath{$\xi$}})$
	\item[(ii)] $\sum_{i=1}^s\Phi_{\xi_{i'}^*, i'\neq i}(\xi_i^*,\xi_i)\geq 0,\,\, \forall\, \xi_i \in \Xi_i,\,  i=1, \dots, s$
	\item[(iii)] $\Phi_{\xi_{i'}^*, i'\neq i}(\xi_i^*,\xi_i)\geq 0,\,\, \forall\, \xi_i \in \Xi_i,\,  i=1, \dots, s$
	\item[(iv)] $\Phi_{\xi_{i'}^*, i'\neq i}(\xi_i^*,\delta_{x_{i}})\geq 0,\,\, \forall\, {x_{i}} \in \mathcal{X}_i,\,  i=1, \dots, s$.
\end{enumerate}
	\item[b)] Let $\mbox{\boldmath{$\xi$}}^*=(\xi_1^*, \dots, \xi_s^*)$ minimize $\phi(\mbox{\boldmath{$\xi$}})$. Let ${x_{i}}$ be a support point of $\xi_i^*$, $i \in \{1, \dots, s\}$. Then $\Phi_{\xi_{i'}^*, i'\neq i}(\xi_i^*,\delta_{x_{i}})= 0$.
	\item[c)] Let $\mbox{\boldmath{$\xi$}}^*=(\xi_1^*, \dots, \xi_s^*)$ minimize $\phi(\mbox{\boldmath{$\xi$}})$. Then the point $(\mbox{\boldmath{$\xi$}}^*,\mbox{\boldmath{$\xi$}}^*)$ is a saddle point of $\Phi$ in that 
\begin{equation}\label{sp11}
\Phi(\mbox{\boldmath{$\xi$}}^*,\mbox{\boldmath{$\xi$}})\geq 0= \Phi(\mbox{\boldmath{$\xi$}}^*,\mbox{\boldmath{$\xi$}}^*)\geq \Phi(\tilde{\mbox{\boldmath{$\xi$}}},\mbox{\boldmath{$\xi$}}^*), \quad \forall\, \mbox{\boldmath{$\xi$}}, \tilde{\mbox{\boldmath{$\xi$}}} \in \Xi
\end{equation}
and the point $(\xi_i^*,\xi_i^*)$ is a saddle point of $\Phi_{\xi_{i'}^*, i'\neq i}$ in that 
\begin{equation}\label{sp22}
\Phi_{\xi_{i'}^*, i'\neq i}(\xi_i^*,\xi_i)\geq 0= \Phi_{\xi_{i'}^*, i'\neq i}(\xi_i^*,\xi_i^*)\geq \Phi_{\xi_{i'}^*, i'\neq i}(\tilde{\xi}_i,\xi_i^*), \quad \forall\, \xi_i, \tilde{\xi}_i \in \Xi_i,
\end{equation}
for all $i=1, \dots, s$.
\end{enumerate}
\end{thm}
\begin{proof}
\begin{enumerate}
	\item[a)] 
	
\textbf{(i)$\Leftrightarrow$(ii)}: 

For this proof we present designs in form of row vectors $\xi_i=(w_{i1}, \dots, w_{ik_i})$, where $w_{it}\geq 0$ is the weight of observations at $x_{it}$, the $t$-th point of the experimental region $\mathcal{X}_i$, 
$t=1, \dots, k_i$, $\sum_{t=1}^{k_i}w_{it}=1$ (see also \cite{boy}, ch.~7). Then $\mbox{\boldmath{$\xi$}}=(\xi_1, \dots, \xi_s)$ is the full (row) vector of all weights of observations at all points of all experimental regions. 

We use the notations $\nabla_{\xi_i}\phi$ for the gradient of $\phi$ with respect to $\xi_i$: $\nabla_{\xi_i}\phi=\left(\frac{\partial\phi}{\partial w_{it}}\right)_{t=1, \dots, k_i}$, and $\nabla_{\xi}\phi$ for the gradient of $\phi$ with respect to $\mbox{\boldmath{$\xi$}}$, which means $\nabla_{\xi}\phi=(\nabla_{\xi_1}\phi, \dots, \nabla_{\xi_s}\phi)$. (Gradients $\nabla_{\xi_i}\phi$ and $\nabla_{\xi}\phi$ are row vectors).

According to convex optimization theory (see e.\,g. \cite{boy}, ch.~4) $\mbox{\boldmath{$\xi$}}^*$ minimizes $\phi$ iff $\Phi(\mbox{\boldmath{$\xi$}}^*,\mbox{\boldmath{$\xi$}})\geq 0$ for all $\mbox{\boldmath{$\xi$}} \in \Xi$. Then the equivalence of (i) and (ii) follows from
\begin{eqnarray*}
\Phi(\mbox{\boldmath{$\xi$}}^*, \mbox{\boldmath{$\xi$}})&=& \nabla_{\xi}\phi(\mbox{\boldmath{$\xi$}}^*)(\mbox{\boldmath{$\xi$}}-\mbox{\boldmath{$\xi$}}^*)^\top\\
&=& \sum_{i=1}^s\nabla_{\xi_i}\phi(\mbox{\boldmath{$\xi$}}^*)(\xi_i-\xi_i^*)^\top\\
&=& \sum_{i=1}^s\Phi_{\xi_{i'}^*, i'\neq i}(\xi_i^*, \xi_i).
\end{eqnarray*}

\textbf{(ii)$\Leftrightarrow$(iii)}: (iii)$\Rightarrow$(ii) Straightforward

(ii)$\Rightarrow$(iii): Let $\exists$ $\tilde{\xi}_{i_1} \in \mathcal{X}_i$ with $\Phi_{\xi_{i'}^*, i'\neq i}(\xi_{i_1}^*,\tilde{\xi}_{i_1})<0$. Let $\xi_{i_1}=\tilde{\xi}_{i_1}$ 
and $\xi_i=\xi_i^*$, $\forall\, i\neq i_1$. Then for all $i\neq i_1$ we have $\Phi_{\xi_{i'}^*, i'\neq i}(\xi_i^*,\xi_i)=\Phi_{\xi_{i'}^*, i'\neq i}(\xi_i^*,\xi_i^*)=0$, which results in
\begin{eqnarray*}
\sum_{i=1}^s\Phi_{\xi_{i'}^*, i'\neq i}(\xi_i^*,\xi_i)&=& \Phi_{\xi_{i'}^*, i'\neq i}(\xi_{i_1}^*,\tilde{\xi}_{i_1})+\sum_{i\in\{1, \dots, s\}\backslash i_1}\Phi_{\xi_{i'}^*, i'\neq i}(\xi_i^*,\xi_i^*)\\
&=& \Phi_{\xi_{i'}^*, i'\neq i}(\xi_{i_1}^*,\tilde{\xi}_{i_1})<0.
\end{eqnarray*}

\textbf{(iii)$\Leftrightarrow$(iv)}: (iii)$\Rightarrow$(iv) Straightforward

(iv)$\Rightarrow$(iii) 
Let $x_{i}=x_{it}$ be the $t$-th point in $\mathcal{X}_i$, $t = 1, \dots, k_i$.
Then the one-point design with all observations at $x_{i}$ is given by $\delta_{x_{it}}=\mathbf{e}_{t}$, where $\mathbf{e}_{t}$ is the $t$-th unit (row) vector of length $k_i$. A design $\xi_i$ can be written as $\xi_i=\sum_{t=1}^{k_i}w_{it}$. Then
the directional derivative of $\phi$ at $\xi_i^*$ in direction of $\xi_i$ can be presented in form
\begin{equation*}
\Phi_{\xi_{i'}^*, i'\neq i}(\xi_i^*,\xi_i)=\sum_{t=1}^{k_i}w_{it}\nabla_{\xi_i}\phi(\xi^*)(\mathbf{e}_{t}-\xi_i^*)^\top,
\end{equation*}
which results in
\begin{equation*}
\Phi_{\xi_{i'}^*, i'\neq i}(\xi_i^*,\xi_i)=\sum_{t=1}^{k_i}w_{it}\Phi_{\xi_{i'}^*, i'\neq i}(\xi_i^*,\delta_{x_{it}})\geq 0.
\end{equation*}


	\item[b)] Let the support point $x_{i}=x_{it'}$ be the $t'$-th point in $\mathcal{X}_i$, $t' \in 1, \dots, k_i$. Then for $\xi_i^*=(w_{i1}^*, \dots, w_{ik_i}^*)$ we have $w_{it'}^*>0$ and 
\begin{equation*}
\Phi_{\xi_{i'}^*, i'\neq i}(\xi_i^*,\xi_i^*)=\sum_{t=1}^{k_i}w_{it}^*\Phi_{\xi_{i'}^*, i'\neq i}(\xi_i^*,\delta_{x_{it}}).
\end{equation*}	
Let $\Phi_{\xi_{i'}^*, i'\neq i}(\xi_i^*,\delta_{x_{it'}})> 0$. Then since
$\Phi_{\xi_{i'}^*, i'\neq i}(\xi_i^*,\delta_{x_{it}})\geq 0$, $t=1, \dots, k_i$, we obtain $\Phi_{\xi_{i'}^*, i'\neq i}(\xi_i^*,\xi_i^*)>0$.
	
	\item[c)] The left-hand sides of both \eqref{sp11} and \eqref{sp22} are straightforward. From formula \eqref{dd1} and convexity of $\phi$ we obtain 
	\begin{equation*}
\Phi(\tilde{\mbox{\boldmath{$\xi$}}},\mbox{\boldmath{$\xi$}}^*)\leq \phi(\mbox{\boldmath{$\xi$}}^*)-\phi(\tilde{\mbox{\boldmath{$\xi$}}}),
\end{equation*}
which is non-positive for optimal $\mbox{\boldmath{$\xi$}}^*$ and all $\tilde{\mbox{\boldmath{$\xi$}}} \in \Xi$.
Similarly using formula \eqref{pdd1} we obtain the right-hand side of \eqref{sp22}.
\end{enumerate}
\end{proof}

\subsection{Optimality conditions based on moment matrices}\label{22}

We use the notation $\mathbf{M}_i(\xi_i)$ for a matrix which characterizes a design $\xi_i$. We assume $\mathbf{M}_i(\xi_i)$ to satisfy the condition
\begin{equation}\label{eq0}
\mathbf{M}_i(\xi_i)=\int_{\mathcal{X}_i}\mathbf{M}_i(\delta_{x_i})\xi_i(\textrm{d}x_i)
\end{equation}
for all $i=1, \dots, s$.
We call this matrix \textit{moment matrix} of a design $\xi_i$ (see e.\,g. \cite{puk}). 
$\mathcal{M}_i$ denotes the set of all moment matrices $\mathbf{M}_i(\xi_i)$, $\xi_i \in \Xi_i$. For $\mathbb{M}(\mbox{\boldmath{$\xi$}})=(\mathbf{M}_1(\xi_1), \dots, \mathbf{M}_s(\xi_s))$ , $\mathcal{M}$ denotes the set of all $\mathbb{M}(\mbox{\boldmath{$\xi$}})$, $\mbox{\boldmath{$\xi$}} \in \Xi$. Then $\mathcal{M}=\times_{i=1}^s\mathcal{M}_i$ and $\mathcal{M}$ is convex.
$\phi: \mathcal{M} \rightarrow (-\infty; \infty]$ is a design criterion.
$\Phi(\mathbb{M},\tilde{\mathbb{M}})$ denotes the directional derivative of $\phi$ at $\mathbb{M}$ in direction of $\tilde{\mathbb{M}}$:
\begin{equation}\label{dd}
\Phi(\mathbb{M},\tilde{\mathbb{M}})=\lim_{\alpha\, \searrow\, 0}\frac{1}{\alpha}\left(\phi((1-\alpha)\mathbb{M}+\alpha\tilde{\mathbb{M}})-\phi(\mathbb{M})\right).
\end{equation}
We define the partial directional derivative of $\phi$ at $\mathbf{M}_i$ in direction of $\tilde{\mathbf{M}}_i$ 
as follows:
\begin{equation}\label{pdd}
\Phi_{\mathbf{M}_{i'}, i'\neq i}(\mathbf{M}_i,\tilde{\mathbf{M}}_i)=\Phi(\mathbb{M},\breve{\mathbb{M}}),
\end{equation}
where $\breve{\mathbb{M}}_{i'}=\mathbb{M}_{i'}$, $i'=1, \dots, s$, $i'\neq i$, and $\breve{\mathbb{M}}_{i}=\tilde{\mathbb{M}}_{i}$.

\begin{thm}\label{oc1}
Let $\phi: \mathcal{M} \rightarrow (-\infty; \infty]$ be convex and differentiable.
\begin{enumerate}
	\item[a)] The following statements are equivalent:
	\begin{enumerate}
	\item[(i)] $\mbox{\boldmath{$\xi$}}^*=(\xi_1^*, \dots, \xi_s^*)$ minimizes $\phi(\mathbb{M}(\mbox{\boldmath{$\xi$}}))$
	\item[(ii)] $\sum_{i=1}^s\Phi_{\mathbf{M}_{i'}(\xi_{i'}^*), i'\neq i}(\mathbf{M}_i(\xi_i^*),\mathbf{M}_i(\xi_i))\geq 0,\,\, \forall\, \xi_i \in \Xi_i,\,  i=1, \dots, s$
	\item[(iii)] $\Phi_{\mathbf{M}_{i'}(\xi_{i'}^*), i'\neq i}(\mathbf{M}_i(\xi_i^*),\mathbf{M}_i(\xi_i))\geq 0,\,\, \forall\, \xi_i \in \Xi_i,\,  i=1, \dots, s$
	\item[(iv)] $\Phi_{\mathbf{M}_{i'}(\xi_{i'}^*), i'\neq i}(\mathbf{M}_i(\xi_i^*),\mathbf{M}_i(\delta_{x_i}))\geq 0,\,\, \forall\, x_i \in \mathcal{X}_i,\,  i=1, \dots, s$.
\end{enumerate}
	\item[b)] Let $\mbox{\boldmath{$\xi$}}^*=(\xi_1^*, \dots, \xi_s^*)$ minimize $\phi(\mathbb{M}(\mbox{\boldmath{$\xi$}}))$. Let $x_i$ be a support point of $\xi_i^*$, $i \in \{1, \dots, s\}$. Then $\Phi_{\mathbf{M}_{i'}(\xi_{i'}^*), i'\neq i}(\mathbf{M}_i(\xi_i^*),\mathbf{M}_i(\delta_{x_i}))= 0$.
	\item[c)] Let $\mbox{\boldmath{$\xi$}}^*=(\xi_1^*, \dots, \xi_s^*)$ minimize $\phi(\mathbb{M}(\mbox{\boldmath{$\xi$}}))$. Then the point $(\mathbb{M}(\mbox{\boldmath{$\xi$}}^*),\mathbb{M}(\mbox{\boldmath{$\xi$}}^*))$ is a saddle point of $\Phi$ in that 
\begin{equation}\label{sp1}
\Phi(\mathbb{M}(\mbox{\boldmath{$\xi$}}^*),\mathbb{M}(\mbox{\boldmath{$\xi$}}))\geq 0= \Phi(\mathbb{M}(\mbox{\boldmath{$\xi$}}^*),\mathbb{M}(\mbox{\boldmath{$\xi$}}^*))\geq \Phi(\mathbb{M}(\tilde{\mbox{\boldmath{$\xi$}}}),\mathbb{M}(\mbox{\boldmath{$\xi$}}^*)), \quad \forall\, \mbox{\boldmath{$\xi$}}, \tilde{\mbox{\boldmath{$\xi$}}} \in \Xi
\end{equation}
and the point $(\mathbf{M}_i(\xi_i^*),\mathbf{M}_i(\xi_i^*))$ is a saddle point of $\Phi_{\mathbf{M}_{i'}(\xi_{i'}^*), i'\neq i}$ in that
\begin{eqnarray}\label{sp2}
\hspace*{-0.2cm}\Phi_{\mathbf{M}_{i'}(\xi_{i'}^*), i'\neq i}(\mathbf{M}_i(\xi_i^*),\mathbf{M}_i(\xi_i))\geq 0&=& \Phi_{\mathbf{M}_{i'}(\xi_{i'}^*), i'\neq i}(\mathbf{M}_i(\xi_i^*),\mathbf{M}_i(\xi_i^*))\nonumber \\
&&\hspace*{-0.6cm}\geq\, \, \, \Phi_{\mathbf{M}_{i'}(\xi_{i'}^*), i'\neq i}(\mathbf{M}_i(\tilde{\xi}_i),\mathbf{M}_i(\xi_i^*)), \quad \forall\, \xi_i, \tilde{\xi}_i \in \Xi_i,
\end{eqnarray}
for all $i=1, \dots, s$.
\end{enumerate}
\end{thm}
\begin{proof}
\begin{enumerate}
	\item[a)] 
	
\textbf{(i)$\Leftrightarrow$(ii)}: 

We use for the gradients of $\phi$ with respect to $\mathbf{M}_i$ and $\mathbb{M}$ the notation
\begin{equation*}
\nabla_{M_i}\phi=\left(\frac{\partial\phi}{\partial m_{kl}}\right)_{k,l}, \quad \mathbf{M}_i=( m_{kl})_{k,l}
\end{equation*}
and
\begin{equation*}
\nabla_{M}\phi=\left(\frac{\partial\phi}{\partial m_{kl}}\right)_{k,l}, \quad \mathbb{M}=( m_{kl})_{k,l},
\end{equation*}
respectively.

$\mbox{\boldmath{$\xi$}}^*$ minimizes $\phi$ iff $\Phi(\mathbb{M}(\mbox{\boldmath{$\xi$}}^*),\mathbb{M}(\mbox{\boldmath{$\xi$}}))\geq 0$ for all $\mbox{\boldmath{$\xi$}} \in \Xi$. The directional derivative can be computed by formula
\begin{equation}\label{dd2}
\Phi(\mathbb{M},\tilde{\mathbb{M}})=\frac{\partial\phi}{\partial \alpha}\left((1-\alpha)\mathbb{M}+\alpha\tilde{\mathbb{M}}\right)|_{\alpha=0}.
\end{equation}
Then using some matrix differentiation rules (see e.\,g. \cite{seb}, ch.~17) we receive
\begin{eqnarray*}
\Phi(\mathbb{M}(\mbox{\boldmath{$\xi$}}^*),\mathbb{M}(\mbox{\boldmath{$\xi$}}))&=&\textrm{tr}\left(\nabla_{M}\phi(\mathbb{M}(\mbox{\boldmath{$\xi$}}^*))(\mathbb{M}(\mbox{\boldmath{$\xi$}})-\mathbb{M}(\mbox{\boldmath{$\xi$}}^*))^\top\right)\\
&=& \sum_{i=1}^s\textrm{tr}\left(\nabla_{M_i}\phi(\mathbb{M}(\mbox{\boldmath{$\xi$}}^*))(\mathbf{M}_i(\xi_i)-\mathbf{M}_i(\xi_i^*))\right)\\
&=& \sum_{i=1}^s\Phi_{\mathbf{M}_{i'}(\xi_{i'}^*), i'\neq i}(\mathbf{M}_i(\xi_i^*),\mathbf{M}_i(\xi_i)),
\end{eqnarray*}
which implies the equivalence of (i) and (ii).

\textbf{(ii)$\Leftrightarrow$(iii)}: (iii)$\Rightarrow$(ii) Straightforward

(ii)$\Rightarrow$(iii): Let $\exists$ $\tilde{\xi}_{i_1}$ with $\Phi_{\mathbf{M}_{i'}(\xi_{i'}^*), i'\neq i}(\mathbf{M}_{i_1}(\xi_{i_1}^*),\mathbf{M}_{i_1}(\tilde{\xi}_{i_1}))<0$. Then for $\xi_{i_1}=\tilde{\xi}_{i_1}$ 
and $\xi_i=\xi_i^*$, $\forall\, i\neq i_1$, 
we obtain
\begin{eqnarray*}
\sum_{i=1}^s\Phi_{\mathbf{M}_{i'}(\xi_{i'}^*), i'\neq i}(\mathbf{M}_i(\xi_i^*),\mathbf{M}_i(\xi_i))&=& \Phi_{\mathbf{M}_{i'}(\xi_{i'}^*), i'\neq i}(\mathbf{M}_{i_1}(\xi_{i_1}^*),\mathbf{M}_{i_1}(\tilde{\xi}_{i_1}))\\
&&+\sum_{i\in\{1, \dots, s\}\backslash i_1}\Phi_{\mathbf{M}_{i'}(\xi_{i'}^*), i'\neq i}(\mathbf{M}_i(\xi_i^*),\mathbf{M}_i(\xi_i))\\
&=& \Phi_{\mathbf{M}_{i'}(\xi_{i'}^*), i'\neq i}(\mathbf{M}_{i_1}(\xi_{i_1}^*),\mathbf{M}_{i_1}(\tilde{\xi}_{i_1}))<0.
\end{eqnarray*}

\textbf{(iii)$\Leftrightarrow$(iv)}: (iii)$\Rightarrow$(iv) Straightforward

(iv)$\Rightarrow$(iii) The directional derivative of $\phi$ at $\mathbf{M}_i$ in direction of $\tilde{\mathbf{M}}_i$ is linear in the second argument:
\begin{equation*}
\Phi_{\mathbf{M}_{i'}, i'\neq i}(\mathbf{M}_i,\tilde{\mathbf{M}}_i)=\textrm{tr}\left(\nabla_{M_i}\phi(\mathbb{M})(\tilde{\mathbf{M}}_i-\mathbf{M}_i)\right).
\end{equation*}
Then using formula \eqref{eq0} we obtain
\begin{equation}\label{eq1}
\Phi_{\mathbf{M}_{i'}(\xi_{i'}^*), i'\neq i}(\mathbf{M}_i(\xi_i^*),\mathbf{M}_i(\xi_i))=\int_{\mathcal{X}_i}\Phi_{\mathbf{M}_{i'}(\xi_{i'}^*), i'\neq i}(\mathbf{M}_i(\xi_i^*),\mathbf{M}_i(\delta_{x_i}))\xi_i(\textrm{d}x_i)
\end{equation}
for each $\xi_i \in \Xi_i$.

	\item[b)] The result follows from formula \eqref{eq1}, $\Phi_{\mathbf{M}_{i'}(\xi_{i'}^*), i'\neq i}(\mathbf{M}_i(\xi_i^*),\mathbf{M}_i(\delta_{x_i}))\geq 0$, for all $x_i \in \mathcal{X}_i$, and $\Phi_{\mathbf{M}_{i'}(\xi_{i'}^*), i'\neq i}(\mathbf{M}_i(\xi_i^*),\mathbf{M}_i(\xi_i^*))=0$.
	
	\item[c)] The left-hand sides of both \eqref{sp1} and \eqref{sp2} are straightforward. From convexity of $\phi$ and formula \eqref{dd} we obtain 
	\begin{equation*}
\Phi(\mathbb{M}(\tilde{\mbox{\boldmath{$\xi$}}}),\mathbb{M}(\mbox{\boldmath{$\xi$}}^*))\leq \phi(\mathbb{M}(\mbox{\boldmath{$\xi$}}^*))-\phi(\mathbb{M}(\tilde{\mbox{\boldmath{$\xi$}}})), \quad \forall\, \tilde{\mbox{\boldmath{$\xi$}}} \in \Xi,
\end{equation*}
which implies the right-hand side of \eqref{sp1}. Similarly using formula \eqref{pdd} we obtain the right-hand side of \eqref{sp2}.
\end{enumerate}
\end{proof}

\section{Optimal Designs in Multiple-Group RCR Models}

We consider multiple-group RCR models in which $N$ observational units are assigned to $s$ groups: $n_i$ observational units in the $i$-th group, $\sum_{i=1}^sn_i=N$. 
The group sizes and the group allocation of observational units are fixed. 
Experimental designs are assumed to be the same for all observational units within one group (group-design): $m_i$ observations per unit in design points $x_{ih}$, $h=1, \dots, m_i$,  in group $i$. However, for units from different groups experimental designs are in general not the same: $m_{i'}\neq m_{i''}$ and (or) $x_{i'h}\neq x_{i''h}$, $i'\neq i''$.

Note that the experimental settings $x_{i1}, \dots, x_{im_i}$ in group $i$ are not necessarily all distinct (repeated measurements are not excluded).

Note also that observational units (often called individuals in the literature) are usually expected to be people, animals or plants. However, they may also be studies, centers, clinics, plots, etc.

\subsection{Model specification and estimation of unknown parameters}

In multiple-group random coefficient regression models the $h$-th observation of the $j$-th observational unit in the $i$-th group is given by the following $l$-dimensional random column vector
\begin{equation}\label{f2}
\hspace*{-0.25cm}\mathbf{Y}_{ijh}= \mathbf{G}_i(x_{ih})\mbox{\boldmath{$\beta$}}_{ij} + \mbox{\boldmath{$\varepsilon$}}_{ijh},\quad x_{ih} \in \mathcal{X}_i,\quad h=1, \dots, m_i,\quad j=1,\dots, n_i,\quad i=1,\dots, s,
\end{equation}
where 
$\mathbf{G}_i$ denotes a group-specific ($l\times p$) matrix of known regression functions in group $i$ (in particular case $l=1$: $\mathbf{G}_i=\mathbf{f}_i^\top$, where $\mathbf{f}_i$ is a $p$-dimensional column vector of regression functions), experimental settings $x_{ih}$ come from some experimental region $\mathcal{X}_i$,
$\mbox{\boldmath{$\beta $}}_{ij}=(\beta_{ij1}, \dots, \beta_{ijp})^\top$ are unit-specific random parameters with unknown mean $\mbox{\boldmath{$\beta$}}_0$ and given ($p\times p$) covariance matrix $\mathbf{D}_i$, $\mbox{\boldmath{$\varepsilon$}}_{ijh}$ denote column vectors of observational errors with zero mean and non-singular ($l\times l$) covariance matrix $\mathbf{\Sigma}_i$. All observational errors and all random parameters are assumed to be uncorrelated.


For the vector $\mathbf{Y}_{ij}=(\mathbf{Y}_{ij1}^\top, ..., \mathbf{Y}_{ijm_i}^\top)^\top$ of observations at the $j$-th observational unit in the $i$-th group we obtain
\begin{equation}\label{f1}
\mathbf{Y}_{ij}= \mathbf{F}_i\mbox{\boldmath{$\beta$}}_{ij} + \mbox{\boldmath{$\varepsilon$}}_{ij},\quad j=1,\dots, n_i, \quad i=1,\dots, s,
\end{equation}
where $\mathbf{F}_i=(\mathbf{G}^\top_i(x_{i1}), ..., \mathbf{G}^\top_i(x_{im_i}))^\top$ is the design matrix in group $i$ and $\mbox{\boldmath{$\varepsilon$}}_{ij}=(\mbox{\boldmath{$\varepsilon$}}_{ij1}^\top, ..., \mbox{\boldmath{$\varepsilon$}}_{ijm_i}^\top)^\top$. 
Then the vector $\mathbf{Y}_i=(\mathbf{Y}_{i1}^\top, \dots, \mathbf{Y}_{in_i}^\top)^\top$ of all observations in group $i$ is given by
\begin{equation}
\mathbf{Y}_i=\left(\mathbb{I}_{n_i}\otimes \mathbf{F}_i\right)\mbox{\boldmath{$\beta$}}_i+\mbox{\boldmath{$\varepsilon$}}_i,\quad i=1, \dots s,
\end{equation}
where $\mbox{\boldmath{$\beta$}}_i=(\mbox{\boldmath{$\beta$}}_{i1}^\top, \dots, \mbox{\boldmath{$\beta$}}_{in_i}^\top)^\top$, $\mbox{\boldmath{$\varepsilon$}}_i=(\mbox{\boldmath{$\varepsilon$}}_{i1}^\top, \dots, \mbox{\boldmath{$\varepsilon$}}_{in_i}^\top)^\top$,  $\mathbb{I}_{n_i}$ is the $n_i\times n_i$ identity matrix and ''$\otimes$'' denotes the Kronecker product,
and the total vector $\mathbf{Y}=(\mathbf{Y}_1^\top, \dots, \mathbf{Y}_s^\top)^\top$ of all observations in all groups results in
\begin{equation}
\mathbf{Y}=\left(\begin{array}{c}\mathds{1}_{n_1}\otimes \mathbf{F}_1 \\ \dots \\ \mathds{1}_{n_s}\otimes \mathbf{F}_s \end{array}\right)\mbox{\boldmath{$\beta$}}_0+\tilde{\mbox{\boldmath{$\varepsilon$}}}
\end{equation}
with
\begin{equation*}
\tilde{\mbox{\boldmath{$\varepsilon$}}}=\textrm{block-diag}\left(\mathbb{I}_{n_1}\otimes \mathbf{F}_1, \dots, \mathbb{I}_{n_s}\otimes \mathbf{F}_s\right)\mathbf{b}+\mbox{\boldmath{$\varepsilon$}},
\end{equation*}
where $\textrm{block-diag}(\mathbf{A}_1, \dots, \mathbf{A}_s)$ is the block-diagonal matrix with blocks $\mathbf{A}_1, \dots, \mathbf{A}_s$, $\mathbf{b}=\mbox{\boldmath{$\beta$}}-\mathds{1}_{N}\otimes\mbox{\boldmath{$\beta$}}_0$ for $\mbox{\boldmath{$\beta$}}=(\mbox{\boldmath{$\beta$}}_1^\top, \dots, \mbox{\boldmath{$\beta$}}_{s}^\top)^\top$, $\mbox{\boldmath{$\varepsilon$}}=(\mbox{\boldmath{$\varepsilon$}}_1^\top, \dots, \mbox{\boldmath{$\varepsilon$}}_s^\top)^\top$ and $\mathds{1}_{n_i}$ is the column vector of length $n_i$ with all entries equal to $1$.

Using Gauss-Markov theory we obtain the following best linear unbiased estimator for the mean parameters $\mbox{\boldmath{$\beta $}}_0$:
\begin{eqnarray}\label{blue}
\hat{\mbox{\boldmath{$\beta $}}}_0 =\left[\sum_{i=1}^sn_i((\tilde{\mathbf{F}}_i^\top \tilde{\mathbf{F}}_i)^{-1}+\mathbf{D}_i)^{-1}\right]^{-1}\sum_{i=1}^sn_i((\tilde{\mathbf{F}}_i^\top \tilde{\mathbf{F}}_i)^{-1}+\mathbf{D}_i)^{-1}\hat{\mbox{\boldmath{$\beta $}}}_{0,i},
\end{eqnarray}
where
$\hat{\mbox{\boldmath{$\beta $}}}_{0,i}= (\tilde{\mathbf{F}}_i^\top \tilde{\mathbf{F}}_i)^{-1}\tilde{\mathbf{F}}_i^\top \tilde{\bar{\mathbf{Y}}}_i$ is the estimator based only on observations in the  $i$-th group, $\tilde{\mathbf{F}}_i=(\mathbb{I}_{n_i}\otimes \mathbf{\Sigma}_i^{-1/2})\mathbf{F}_i$ and $\tilde{\bar{\mathbf{Y}}}_i=(\mathbb{I}_{n_i}\otimes \mathbf{\Sigma}_i^{-1/2})\bar{\mathbf{Y}}_i$ (for the mean observational vector $\bar{\mathbf{Y}}_i=\frac{1}{n_i}\sum_{j=1}^{n_i}\mathbf{Y}_{ij}$ in the $i$-th group and the symmetric positive definite matrix $\mathbf{\Sigma}_i^{1/2}$ with the property $\mathbf{\Sigma}_i=\mathbf{\Sigma}_i^{1/2}\mathbf{\Sigma}_i^{1/2}$) are the transformed design matrix and the transformed mean observational vector with respect to the covariance structure of observational errors.

Note that BLUE \eqref{blue} exists only if all matrices $\tilde{\mathbf{F}}_i^\top \tilde{\mathbf{F}}_i$ are non-singular. Therefore, we restrict ourselves on the case where design matrices $\mathbf{F}_i$ are of full column rank for all $i$.

The covariance matrix of the best linear unbiased estimator $\hat{\mbox{\boldmath{$\beta $}}}_0$ is given by
\begin{equation}\label{cov}
\textrm{Cov}\left(\hat{\mbox{\boldmath{$\beta $}}}_0\right) = \left[\sum_{i=1}^sn_i((\tilde{\mathbf{F}}_i^\top \tilde{\mathbf{F}}_i)^{-1}+\mathbf{D}_i)^{-1}\right]^{-1}.
\end{equation}

In \cite{fedo} similar results were obtained for the multi-center trials models.

\subsection{Design criteria}

We define an exact design in group $i$ as
\[ \xi_i= \left( \begin{array}{c} 
    x_{i1},\dots , x_{ik_i} \\ 
    m_{i1},\dots , m_{ik_i}
\end{array} \right),\]
where $x_{i1},\dots , x_{ik_i}$ are the (distinct) experimental settings in $\mathcal{X}_i$ with the related numbers of observations $m_{i1},\dots , m_{ik_i}$, $\sum_{h=1}^{k_i}{m_{ih}}=m_i$. For analytical purposes we also introduce approximate designs:
\[ \xi_i= \left( \begin{array}{c} 
    x_{i1},..., x_{ik_i} \\ 
    w_{i1},..., w_{ik_i}
\end{array} \right),\]
where $w_{ih}\geq 0$ denotes the weight of observations at $x_{ih}$, $h=1, \dots k_i$, and $\sum_{h=1}^{k_i}{w_{ih}}=1$.

We will use the following notation for the moment matrix in group $i$:
\begin{equation}\label{info}
\mathbf{M}_i(\xi_i)=\sum_{h=1}^{k_i}w_{ih}\,\tilde{\mathbf{G}}_i(x_{ih})^\top\tilde{\mathbf{G}}_i(x_{ih}),
\end{equation} 
where $\tilde{\mathbf{G}}_i=\mathbf{\Sigma}_i^{-1/2}\mathbf{G}_i$. 
For exact designs we have $w_{ih}=m_{ih}/m_i$ and
\begin{equation*}
\mathbf{M}_i(\xi_i)=\frac{1}{m_i}\tilde{\mathbf{F}}_i^\top\tilde{\mathbf{F}}_i.
\end{equation*}
We will also use the notation $\mathbf{\Delta}_i=m_i\mathbf{D}_i$ for the adjusted dispersion matrix of random effects in group $i$.

Then we extend the definition of the variance-covariance matrix \eqref{cov} with respect to approximate designs: 
\begin{equation}\label{cov1}
\textrm{Cov}_{\xi_1, \dots, \xi_s} = \left[\sum_{i=1}^sn_im_i\left(\mathbf{M}_i(\xi_i)^{-1}+\mathbf{\Delta}_i\right)^{-1}\right]^{-1}.
\end{equation}

Further we focus on the linear (\textit{L}-) and determinant (\textit{D}-) criteria for estimation of the mean parameters $\mbox{\boldmath{$\beta$}}_0$. The linear criterion is defined as
\begin{equation}
\phi_{L} =\mathrm{tr}\,\left[\mathrm{Cov}\left(\mathbf{L} \hat{\mbox{\boldmath{$\beta $}}}_0\right)\right],
\end{equation}  
where $\mathbf{L}$ is some linear transformation of $\mbox{\boldmath{$\beta$}}_0$.
For approximate designs we obtain 
\begin{equation}\label{lcr}
\phi_{L}(\xi_1, \dots, \xi_s)=\mathrm{tr}\left(\left[\sum_{i=1}^sn_im_i\left(\mathbf{M}_i(\xi_i)^{-1}+\mathbf{\Delta}_i\right)^{-1}\right]^{-1}\mathbf{V}\right),
\end{equation}
where $\mathbf{V} = \mathbf{L}^\top\mathbf{L}$.
\begin{bem}
The \textit{A}-, and \textit{c}-criteria for estimation of $\mbox{\boldmath{$\beta$}}_0$ are the particular linear criteria with $\mathbf{V}=\mathbb{I}_p$ and $\mathbf{V}=\mathbf{c}\mathbf{c}^\top$, for a $p$-dimensional real column vector $\mathbf{c}$, respectively.
\end{bem}
If the regression matrices and the experimental regions are the same among all groups: $\mathbf{G}_i=\mathbf{G}$, $\mathcal{X}_i=\mathcal{X}$, $i=1, \dots, s$, the integrated mean squared error (IMSE-) criterion can also be used for multiple group models. We define the IMSE-criterion for estimation of the mean parameters $\mbox{\boldmath{$\beta$}}_0$ as follows:
\begin{equation}\label{imse}
\phi_{IMSE} =\mathrm{tr}\left(\mathrm{E}\left[\int_{\mathcal{X}}\left(\mathbf{G}(x) \hat{\mbox{\boldmath{$\beta $}}}_0\right)\left(\mathbf{G}(x) \hat{\mbox{\boldmath{$\beta $}}}_0\right)^\top\nu(\textrm{d}x)\right]\right),
\end{equation}
where $\nu$ is some suitable measure on the experimental region $\mathcal{X}$, which is typically
chosen to be uniform on $\mathcal{X}$ with $\nu(\mathcal{X})=1$.
IMSE-criterion \eqref{imse} can be recognized as the particular linear criterion with $\mathbf{V} = \int_{\mathcal{X}}\mathbf{G}(x)^\top \mathbf{G}(x)\,\nu(\textrm{d}x)$.

The \textit{D}-criterion is defined as the logarithm of the determinant of the covariance matrix of the estimation, which results in  
\begin{equation}\label{dcr}
\phi_{D}(\xi_1, \dots, \xi_s)=-\mathrm{ln}\,\mathrm{det}\left(\sum_{i=1}^sn_im_i\left(\mathbf{M}_i(\xi_i)^{-1}+\mathbf{\Delta}_i\right)^{-1}\right)
\end{equation}
for approximate designs.

Note that optimal designs depend on the group sizes $n_i$ only via the proportions $n_i/N$, $i=1, \dots,s$, and are, hence, independent of the total number of observational units $N$ itself. This statement is easy to verify by formulas \eqref{lcr} and \eqref{dcr}. 

\subsection{Optimality conditions}

To make use of the equivalence theorems proposed in Section~2 we verify the convexity of the design criteria.
\begin{lem}\label{conv} 
The \textit{L}- and \textit{D}-criteria 
for estimation of the mean parameters $\mbox{\boldmath{$\beta$}}_0$ are convex with respects to $\mathbb{M}(\mbox{\boldmath{$\xi$}})=(\mathbf{M}_1(\xi_1), \dots \mathbf{M}_s(\xi_s))$.
\end{lem}
\begin{proof}
The function $\phi(\mathbf{N})=\mathbf{N}^{-1}$ is matrix-convex for any positive definite matrix $\mathbf{N}$, i.\,e. 
\begin{equation}\label{mc}
\left(\alpha\mathbf{N}_1+(1-\alpha)\mathbf{N}_2\right)^{-1}\leq \alpha\mathbf{N}_1^{-1}+(1-\alpha)\mathbf{N}_2^{-1}
\end{equation}
in Loewner ordering for any $\alpha \in [0,1]$ and all positive definite $\mathbf{N}_1$  and $\mathbf{N}_2$ (see e.\,g. \cite{seb}, ch.~10). 
Since $\phi$ is non-increasing in Loewner ordering, it is easy to verify that $\psi_i(\mathbf{M}_i)=\left(\mathbf{M}_i^{-1}+\mathbf{\Delta}_i\right)^{-1}$ is matrix-concave for any positive definite $\mathbf{M}_i$ (see e.\,g. \cite{ber}, ch.~10). Then $\psi(\mathbf{M}_1, \dots \mathbf{M}_s)=\sum_{i=1}^sn_im_i\psi_i(\mathbf{M}_i)$ is matrix-concave with respects to $\mathbb{M}=(\mathbf{M}_1, \dots \mathbf{M}_s)$. 
The functions $\phi_1(\mathbf{N})=-\mathrm{ln}\,\mathrm{det}(\mathbf{N})$ and $\phi_2(\mathbf{N})=\mathrm{tr}\left(\mathbf{N}^{-1}\mathbf{V}\right)$ are non-increasing in Loewner ordering and convex for any positive definite matrix $\mathbf{N}$ and any positive semi-definite matrix $\mathbf{V}$ as the standard \textit{D}- and \textit{L}-criteria (see e.\,g. \cite{paz}, ch.~4, or \cite{fed2}, ch.~2). Then the functions $\phi_1\circ\psi$ and $\phi_2\circ\psi$ are convex.
\end{proof}
We formulate optimality conditions for criteria \eqref{lcr} and \eqref{dcr} based on the results of Theorem~\ref{oc1}.
\begin{thm}\label{ocl}
Approximate designs $\mbox{\boldmath{$\xi$}}^*=(\xi_1^*,\dots, \xi_s^*)$ are \textit{L}-optimal for estimation of the mean parameters $\mbox{\boldmath{$\beta$}}_0$ iff
\begin{eqnarray}\label{ineq3}
&&\hspace*{-0.5cm}\mathrm{tr}\left\{\tilde{\mathbf{G}}_i(x_i)\left[\mathbf{M}_i(\xi_i^*)^{-1}\left(\mathbf{M}_i(\xi_i^*)^{-1}+\mathbf{\Delta}_i\right)^{-1}\left[\sum_{r=1}^sn_rm_r\left(\mathbf{M}_r(\xi_r^*)^{-1}+\mathbf{\Delta}_r\right)^{-1}\right]^{-1}\mathbf{V}\right.\right.\nonumber \\
&& \left.\left.\cdot\left[\sum_{r=1}^sn_rm_r\left(\mathbf{M}_r(\xi_r^*)^{-1}+\mathbf{\Delta}_r\right)^{-1}\right]^{-1}\left(\mathbf{M}_i(\xi_i^*)^{-1}+\mathbf{\Delta}_i\right)^{-1}\mathbf{M}_i(\xi_i^*)^{-1}\right]\tilde{\mathbf{G}}_i(x_i)^\top\right\}\nonumber \\
&&\hspace*{-0.5cm}\leq
\mathrm{tr}\left\{\mathbf{M}_i(\xi_i^*)^{-1}\left(\mathbf{M}_i(\xi_i^*)^{-1}+\mathbf{\Delta}_i\right)^{-1}\left[\sum_{r=1}^sn_rm_r\left(\mathbf{M}_r(\xi_r^*)^{-1}+\mathbf{\Delta}_r\right)^{-1}\right]^{-1}\mathbf{V}\right. \nonumber \\
&& \left.\cdot\left[\sum_{r=1}^sn_rm_r\left(\mathbf{M}_r(\xi_r^*)^{-1}+\mathbf{\Delta}_r\right)^{-1}\right]^{-1}\left(\mathbf{M}_i(\xi_i^*)^{-1}+\mathbf{\Delta}_i\right)^{-1}\right\}
\end{eqnarray}
for $x_i \in \mathcal{X}_i$, $i=1, \dots, s$.

For support points of $\xi_i^*$ equality holds in \eqref{ineq3}.
\end{thm}
\begin{proof}
We obtain the results using Lemma~\ref{conv} and parts a) (equivalence of (i) and (iv)) and b) of Theorem~\ref{oc1} for the partial directional derivatives
\begin{eqnarray}
\mathbf{\Phi}_{L, \mathbf{M}_{i'}, i'\neq i}(\mathbf{M}_i,\tilde{\mathbf{M}}_{i})&=&-n_im_i\,\mathrm{tr}\left\{\left[\sum_{r=1}^sn_rm_r\left(\mathbf{M}_r^{-1}+\mathbf{\Delta}_r\right)^{-1}\right]^{-1}\left(\mathbf{M}_i^{-1}+\mathbf{\Delta}_i\right)^{-1}\mathbf{M}_i^{-1}(\tilde{\mathbf{M}}_i-\mathbf{M}_i)\right.\nonumber \\
&& \cdot \left.\mathbf{M}_i^{-1}\left(\mathbf{M}_i^{-1}+\mathbf{\Delta}_i\right)^{-1}\left[\sum_{r=1}^sn_rm_r\left(\mathbf{M}_r^{-1}+\mathbf{\Delta}_r\right)^{-1}\right]^{-1}\mathbf{V}\right\}
\end{eqnarray}
for $i=1, \dots, s$.
\end{proof}
\begin{thm}\label{ocd}
Approximate designs $\mbox{\boldmath{$\xi$}}^*=(\xi_1^*,\dots, \xi_s^*)$ are \textit{D}-optimal for estimation of the mean parameters $\mbox{\boldmath{$\beta$}}_0$ iff
\begin{eqnarray}\label{ineq4}
&&\hspace*{-0.5cm}\mathrm{tr}\left\{\tilde{\mathbf{G}}_i(x_i)\left[\mathbf{M}_i(\xi_i^*)^{-1}\left(\mathbf{M}_i(\xi_i^*)^{-1}+\mathbf{\Delta}_i\right)^{-1}\left[\sum_{r=1}^sn_rm_r\left(\mathbf{M}_r(\xi_r^*)^{-1}+\mathbf{\Delta}_r\right)^{-1}\right]^{-1}\right.\right.\nonumber \\
&& \left.\left.\cdot\left(\mathbf{M}_i(\xi_i^*)^{-1}+\mathbf{\Delta}_i\right)^{-1}\mathbf{M}_i(\xi_i^*)^{-1}\right]\tilde{\mathbf{G}}_i(x_i)^\top\right\}\nonumber \\
&&\hspace*{-0.5cm}\leq
\mathrm{tr}\left\{\mathbf{M}_i(\xi_i^*)^{-1}\left(\mathbf{M}_i(\xi_i^*)^{-1}+\mathbf{\Delta}_i\right)^{-1}\left[\sum_{r=1}^sn_rm_r\left(\mathbf{M}_r(\xi_r^*)^{-1}+\mathbf{\Delta}_r\right)^{-1}\right]^{-1}\right. \nonumber \\
&& \left.\cdot\left(\mathbf{M}_i(\xi_i^*)^{-1}+\mathbf{\Delta}_i\right)^{-1}\right\}
\end{eqnarray}
for $x_i \in \mathcal{X}_i$, $i=1, \dots, s$.

For support points of $\xi_i^*$ equality holds in \eqref{ineq4}.
\end{thm}
\begin{proof}
The optimality condition follows from Lemma~\ref{conv} and Theorem~\ref{oc1} with the partial directional derivatives
\begin{eqnarray}
\mathbf{\Phi}_{D, \mathbf{M}_{i'}, i'\neq i}(\mathbf{M}_i,\tilde{\mathbf{M}}_{i})&=&-n_im_i\,\mathrm{tr}\left\{\left[\sum_{r=1}^sn_rm_r\left(\mathbf{M}_r^{-1}+\mathbf{\Delta}_r\right)^{-1}\right]^{-1}\left(\mathbf{M}_i^{-1}+\mathbf{\Delta}_i\right)^{-1}\right.\nonumber \\
&& \hspace*{3.2cm}\cdot \left.\mathbf{M}_i^{-1}(\tilde{\mathbf{M}}_i-\mathbf{M}_i)\mathbf{M}_i^{-1}\left(\mathbf{M}_i^{-1}+\mathbf{\Delta}_i\right)^{-1}\right\}
\end{eqnarray}
for $i=1, \dots, s$.
\end{proof}

Note that the results of Theorems~\ref{ocl} and \ref{ocd} coincide for $l=1$ and $n_1=n_2=1$ with optimality conditions for group-wise identical designs in \cite{sch}, ch.~8.

\subsection{Particular case}

We consider the particular multiple-group models, where the regression matrices $\mathbf{G}_i$, the numbers $m_i$ of observations per observational unit, the covariance matrices $\mathbf{D}_i$ and $\mathbf{\Sigma}_i$ of random effects and observational errors and the experimental regions $\mathcal{X}_i$ are the same among all groups. For these models we have $m_i=m$, $\mathbf{\Delta}_i=\mathbf{\Delta}$ and $\mathbf{M}_i(\xi)=\mathbf{M}(\xi)$, $i=1, \dots, s$. Then \textit{L}- and \textit{D}-criteria \eqref{lcr} and \eqref{dcr} simplify to
\begin{equation}
\phi_{L}(\xi_1, \dots, \xi_s)=\mathrm{tr}\left(\left[m\sum_{i=1}^sn_i\left(\mathbf{M}(\xi_i)^{-1}+\mathbf{\Delta}\right)^{-1}\right]^{-1}\mathbf{V}\right)
\end{equation}
and 
\begin{equation}
\phi_{D}(\xi_1, \dots, \xi_s)=-\mathrm{ln}\,\mathrm{det}\left(m\sum_{i=1}^sn_i\left(\mathbf{M}(\xi_i)^{-1}+\mathbf{\Delta}\right)^{-1}\right).
\end{equation}
We denote by $\xi_L^*$ an optimal design for the classical linear criterion
\begin{equation}\label{lcr1}
\phi_{L}(\xi)=\mathrm{tr}\left(\mathbf{M}(\xi)^{-1}\mathbf{V}\right)
\end{equation}
and $\xi_{D}^*$ is an optimal design for the \textit{D}-criterion in the single-group model 
\begin{equation}\label{dcr1}
\phi_{D}(\xi)=\mathrm{ln}\,\mathrm{det}\left(\mathbf{M}(\xi)^{-1}+\mathbf{\Delta}\right).
\end{equation}
Then it can be easily verified that the designs $\xi_i^*=\xi_L^*$ and $\xi_i^*=\xi_D^*$, $i=1, \dots, s$, satisfy the optimality conditions in Theorems~\ref{ocl} and \ref{ocd}, respectively (see \cite{fed}, ch.~5, for the optimality condition with respect to \textit{D}-criterion \eqref{dcr1}).
\begin{fol}\label{c1} 
\textit{L}-optimal designs in the fixed effects model are \textit{L}-optimal as group-designs in the multiple-group RCR model in which the numbers
of observations $m_i$, the regression matrices $\mathbf{G}_i$, the covariance matrices of random effects and observational errors $\mathbf{D}_i$ and $\mathbf{\Sigma}_i$ and the experimental regions $\mathcal{X}_i$ are the same for all groups.
\end{fol}
\begin{fol}\label{c2} 
\textit{D}-optimal designs in the single-group RCR model are \textit{D}-optimal as group-designs in the multiple-group RCR model in which the numbers
of observations $m_i$, the regression matrices $\mathbf{G}_i$, the covariance matrices of random effects and observational errors $\mathbf{D}_i$ and $\mathbf{\Sigma}_i$ and the experimental regions $\mathcal{X}_i$ are the same for all groups.
\end{fol}
The latter statements are expected results in that all observational units in all groups have same statistical properties and there are no group-specific restrictions on designs. Note that the group sizes have no influence on the designs in this case. 
However, as we will see in Section~4 even for statistically identical observational units optimal designs may depend on the group sizes if the numbers of observations per unit differ from group to group. 

\section{Example: Straight Line Regression}

We consider the two-groups model of general form \eqref{f2} with the regression functions $\mathbf{G}_i(x)=(1,x)^\top, x\in\mathcal{X}_i$, and the design region $\mathcal{X}_i=[0,1]$, $i=1,2$:
\begin{equation}\label{lr}
Y_{ijh}= \mbox{\boldmath{$\beta$}}_{ij1} + \mbox{\boldmath{$\beta$}}_{ij2}x_{ih}+\mbox{\boldmath{$\varepsilon$}}_{ijh},\quad h=1,\dots, m_i,\quad j=1,\dots, n_i.
\end{equation}
The covariance structures of random effects and observational errors are given by $\mathbf{D}_i=\mathrm{diag}(d_{i1},d_{i2})$ and $\Sigma_i=1$ for both groups. Group sizes $n_i$ and numbers observations per unit $m_i$ are in general not the same for the first and the second group.

The left-hand sides of optimality conditions \eqref{ineq3} and \eqref{ineq4} result in parabolas with positive leading terms for model \eqref{lr}.
Then \textit{L}- and \textit{D}-optimal approximate designs have the form
\begin{equation}
\xi_i=\left(\begin{array}{cc}0 & 1 \\ 1-w_{i} & w_{i}\end{array}\right),
\end{equation}
where $w_i$ denotes the weight of observations in point $1$ for the $i$-th group, and the moment matrices are given by
\begin{equation}
\mathbf{M}_i=\left(\begin{array}{cc} 1 & w_{i} \\ w_{i} & w_{i}\end{array}\right).
\end{equation}

\subsection{Random intercept}

We consider first the particular case of model \eqref{lr} in which only the intercept $\mbox{\boldmath{$\beta$}}_{ij1}$ is random, i.\,e. $d_{i2}=0$. We focus on the \textit{A}- and \textit{D}-criteria, which are given by \eqref{lcr} for $\mathbf{V}=\mathbb{I}_2$ and \eqref{dcr}, respectively. The \textit{D}-criterion for the random intercept model is given by
\begin{eqnarray}
\phi_{D}(w_1, w_2)=-\mathrm{ln}\left(\sum_{i=1}^2\frac{n_im_i}{d_{i1}m_i+1}\sum_{i=1}^2\frac{n_im_iw_i(d_{i1}m_i(1-w_i)+1)}{d_{i1}m_i+1}-\left(\sum_{i=1}^2\frac{n_im_iw_i}{d_{i1}m_i+1}\right)^2\right).
\end{eqnarray}
This function achieves for all values of $n_i$ and $m_i$ its minimum at point $w_1^*=w_2^*=0.5$, which coincides with the optimal design in the fixed effects model and in the single-group random intercept model (see \cite{schw2}). 
The \textit{A}-criterion for the random intercept model is given by 
\begin{eqnarray}
\phi_{A}(w_1, w_2)=\frac{\sum_{i=1}^2\frac{n_im_i(d_{i1}m_iw_i(1-w_i)+1+w_i)}{d_{i1}m_i+1}}{\sum_{i=1}^2\frac{n_im_i}{d_{i1}m_i+1}\sum_{i=1}^2\frac{n_im_iw_i(d_{i1}m_i(1-w_i)+1)}{d_{i1}m_i+1}-\left(\sum_{i=1}^2\frac{n_im_iw_i}{d_{i1}m_i+1}\right)^2}.
\end{eqnarray}
In contrast to the \textit{D}-criterion, the \textit{A}-optimal weights in general depend on the group sizes $n_i$ and the numbers of observations $m_i$. Some numerical results for $d_{i1}=1$, $i=1,2$, are presented in Table~1.
\begin{table}
\caption{\textit{A}-optimal designs in random intercept model in dependence on group sizes $n_i$ and numbers of observations $m_i$ for $d_{i1}=1$, $i=1,2$}
\begin{center}
\begin{tabular}{|c|c|c|c|c|c|c|c|c|c|c|}
\hline
Case no. & $n_1$ & $n_2$ & $m_1$ & $m_2$ & $m_1/m_2$ & $w^*_{1}$ & $1-w^*_{1}$ & $w^*_{2}$ & $1-w^*_{2}$ \\
\hline
1 & 1 & 1 & 2 & 8 & 1/4 &  0.298 & 0.702 & 0.450 & 0.550 \\
2 & 1 & 1 & 5 & 5 & 1 &  0.414 & 0.586 & 0.414 & 0.586 \\
3 & 1 & 1 & 8 & 2 & 4 &  0.450 & 0.550 & 0.298 & 0.702 \\
\hline
4 & 1 & 1 & 4 & 16 & 1/4 &  0.300 & 0.700 & 0.450 & 0.550 \\
5 & 1 & 1 & 10 & 10 & 1 &  0.414 & 0.586 & 0.414 & 0.586 \\
6 & 1 & 1 & 16 & 4 & 4 &  0.450 & 0.550 & 0.300 & 0.700 \\
\hline
7 & 1 & 2 & 2 & 8 & 1/4 &  0.256 & 0.744 & 0.439 & 0.561 \\
8 & 1 & 2 & 5 & 5 & 1 &  0.414 & 0.586 & 0.414 & 0.586 \\
9 & 1 & 2 & 8 & 2 & 4 &  0.460 & 0.540 & 0.338 & 0.662  \\
\hline
10 & 1 & 2 & 4 & 16 & 1/4 &  0.258 & 0.742 & 0.439 & 0.561 \\
11 & 1 & 2 & 10 & 10 & 1 &  0.414 & 0.586 & 0.414 & 0.586 \\
12 & 1 & 2 & 16 & 4 & 4 &  0.460 & 0.540 & 0.339 & 0.661 \\
\hline
\end{tabular}
\end{center}
\end{table}
As we can see in the table, if the numbers of observations $m_i$ are the same for both groups: cases 2, 5, 8 and 11, the optimal weight $w_A^*=0.414$ in the fixed effects model retains its optimality for the multiple-group model, which is in accordance with Corollary~\ref{c1}. In these cases optimal designs are independent of the group sizes $n_i$ and the numbers of observations $m_i$. For all other cases the optimal weight $w_i^*$ is smaller (larger) than $w_A^*$ if the number of observations $m_i$ is smaller (larger) than the mean number of observations $(m_1+m_2)/2$. For same group sizes ($n_1=n_2$) optimal designs ''swap places'' if the numbers of observations ''swap places'': the optimal weight $w_1^*$ in case 1 is the same as the optimal weight $w_2^*$ in case 3 (same for cases 4 and 6). This property, however, does not hold for different group sizes (cases 7 and 9 or 10 and 12).  

\subsection{Random slope}

Now we consider the particular case of straight line regression model \eqref{lr} in which only the slope $\mbox{\boldmath{$\beta$}}_{ij2}$ is random: $d_{i1}=0$. For this model the \textit{D}-criterion is given by
\begin{eqnarray}
\phi_{D}(w_1, w_2)=-\mathrm{ln}\left(\sum_{i=1}^2\frac{n_im_iw_i}{d_{i1}m_iw_i+1}\sum_{i=1}^2n_im_i(1-w_i)\right).
\end{eqnarray}
In contrast to the random intercept model, \textit{D}-optimal designs for the random slope depend on the group sizes and the numbers of observations. Numerical results for $d_{i2}=1$, $i=1,2$, are presented in Table~2. 
\begin{table}
\caption{\textit{D}-optimal designs in random slope model in dependence on group sizes $n_i$ and numbers of observations $m_i$ for $d_{i2}=1$, $i=1,2$}
\begin{center}
\begin{tabular}{|c|c|c|c|c|c|c|c|c|c|c|c|c|c|c|}
\hline
Case no. & $n_1$ & $n_2$ & $m_1$ & $m_2$ & $m_1/m_2$ & $w^*_{1}$ & $1-w^*_{1}$ & $w^*_{2}$ & $1-w^*_{2}$ \\
\hline
1 & 1 & 1 & 2 & 8 & 1/4 & 0.725 & 0.275 & 0.181 & 0.819 \\
2 & 1 & 1 & 5 & 5 & 1 &  0.290 & 0.710 & 0.290 & 0.710 \\
3 & 1 & 1 & 8 & 2 & 4 & 0.181 & 0.819 & 0.725 & 0.275 \\
\hline
4 & 1 & 1 & 4 & 16 & 1/4 & 0.579 & 0.421 & 0.145 & 0.855 \\
5 & 1 & 1 & 10 & 10 & 1 & 0.232 & 0.768 & 0.232 & 0.768 \\
6 & 1 & 1 & 16 & 4 & 4 & 0.145 & 0.855 & 0.579 & 0.421 \\
\hline
7 & 1 & 2 & 2 & 8 & 1/4 & 0.823 & 0.177 & 0.206 & 0.794 \\
8 & 1 & 2 & 5 & 5 & 1 & 0.290 & 0.710 & 0.290 & 0.710 \\
9 & 1 & 2 & 8 & 2 & 4 & 0.155 & 0.845 & 0.618 & 0.382 \\
\hline
10 & 1 & 2 & 4 & 16 & 1/4 & 0.651 & 0.349 & 0.163 & 0.837 \\
11 & 1 & 2 & 10 & 10 & 1 & 0.232 & 0.768 & 0.232 & 0.768 \\
12 & 1 & 2 & 16 & 4 & 4 & 0.125 & 0.875 & 0.500 & 0.500 \\
\hline
\end{tabular}
\end{center}
\end{table}
As we can see in the table, if $m_1=m_2$ optimal designs are the same for both groups (cases 2, 5, 8 and 11). However, they depend on the numbers of observations $m_i$ themselves (optimal weights in cases 2 and 8 differ from those in cases 5 and 11). This phenomenon is in accordance with Corollary~\ref{c2} since the \textit{D}-optimal weight $w_D^*$ in the single-group model (which minimizes criterion \eqref{dcr1}) also depends on the number of observations (via matrix $\mathbf{\Delta}$). In contrast to the random intercept model, for the random slope the optimal weight $w_i^*$ is larger (smaller) than $w_D^*$ 
if $m_i$ is smaller (larger) than $(m_1+m_2)/2$.

The minimization of the \textit{A}-criterion for the random slope model leads (with the only exception in case $m_1=m_2$) to a singular design matrix. Therefore, this criterion will not be considered here.

\section{Discussion}

In this work we considered design problems depending on several designs simultaneously. We proposed equivalence theorems based on the assumptions of convexity and differentiability of the optimality criteria. For design problems with finite experimental regions we formulated optimality conditions with respect to the designs themselves (Theorem~1). If the optimality criteria depend on the designs via moment matrices only, optimality conditions are formulated with respect to the moment matrices (Theorem~2).

We applied the proposed optimality conditions to the multiple-group RCR models. If all observational units have the same statistical properties and there are no group-specific design-restrictions, optimal designs in the single-group models are also optimal as group-designs in the multiple-group models. In this case the group sizes have no influence on the designs. However, if the numbers of observations differ from group to group, optimal group-designs may depend on the numbers of observations and the group sizes. This behavior has been illustrated by the example of straight line regression models.

The proposed results solve the problem of optimal designs for the estimation of fixed effects in the multiple-group RCR models. The problem of prediction of random parameters remains, however, open. Also for the computation of optimal approximate and exact designs some new approach has to be developed. Moreover, we assumed fixed numbers of observational units $n_i$ and observations per unit $m_i$. Design optimization with respect to these numbers may be an interesting direction for a future research.

\section*{Acknowledgment}

This research was supported by grant SCHW 531/16 of the German Research Foundation (DFG).
The author is grateful to Norbert Gaffke for fruitful discussions on optimality conditions in Theorem~\ref{oc1} and convexity in Lemma~\ref{conv}. Also discussions with Rainer Schwabe on application in mixed models have been very helpful.

\appendix


\bibliographystyle{natbib}
\bibliography{prus7}

\end{document}